\documentclass[12pt]{article}

\usepackage{mathrsfs}
\usepackage{amsmath}
\usepackage{amsfonts}
\usepackage{amssymb}

\usepackage{amsmath, amsthm, amsfonts, amssymb}
\def\a{\alpha}

\newcommand{\CC}{{\mathbb C}}
\newcommand{\RR}{{\mathbb R}}

\newcommand{\PP}{{\mathbb P}}

\newcommand{\ra}{\rightarrow}
\def \-{\bar}

\newtheorem{theorem}{Theorem}[section]
\newtheorem{lemma}[theorem]{Lemma}
\newtheorem{corollary}[theorem]{Corollary}

\date{}

\begin{document}

\title{\bf Submanifolds of Hermitian symmetric spaces} 

\author{
\ \ Xiaojun Huang\footnote{ Supported in part by National Science
Foundation grant DMS-1363418}, \ \ Yuan Yuan\footnote{ Supported in
part by National Science Foundation grant DMS-1412384}}

\vspace{3cm} \maketitle

\begin{abstract}
 We study the problem of
non-relativity for a complex Euclidean space and a bounded symmetric
domain equipped with their canonical metrics. In particular, we
answer a question raised by Di Scala.

This paper is dedicated to the memory of Salah Baouendi, a great
teacher and a close friend to many of us.

\end{abstract}

\bigskip
\section{Introduction}

Holomorphic isometric embeddings have been studied extensively by
many authors. In the celebrated paper by Calabi \cite{C}, he obtained
the global extendability and rigidity  of a local holomorphic
isometry into a complex space form, among many other important
results. In particular, he proved that any complex space form cannot
be locally isometrically embedded into another complex space form
with a different curvature sign with respect to the canonical
K\"ahler metrics, respectively. In his paper, Calabi introduced the
so called diastasis function and reduced the metric tensor equation to the
functional identity for the diastasis functions. In a later
development \cite{DL1}, Di Scala and Loi generalized Calabi's
non-embeddability  result to the case of Hermitian symmetric spaces
of different types.

On the other hand, Umehara \cite{U} studied an interesting question
whether two complex space forms can share a common submanifold with
the induced metrics. Following Calabi's idea, Umehara proved that
two complex space forms with different curvature signs cannot share
a common K\"ahler submanifold. When two complex manifolds share a
common K\"ahler submanifolds with induced metrics, Di Scala and Loi
 in \cite{DL2} called them to be relatives. Furthermore, Di Scala and Loi proved
that a bounded domain with its associated Bergman metric can not be
a relative to a Hermitian symmetric space of compact type equipped
with the canonical metric. Notice that any irreducible Hermitian
symmetric space of compact type can be holomorphically isometrically
embedded into a complex project space by the classical
Nakagawa-Takagi embedding. Therefore in order to show that a
K\"ahler manifold is not a relative of a projective manifold with
induced metric, it suffices to show that it is not a relative to the
complex projective space with the Fubini-Study metric. Meanwhile it
follows from the result of Umehara \cite{U}, the complex Euclidean
space and the irreducible Hermitian symmetric space of compact type
cannot be relatives. After these studies, it remains to understand
if a complex Euclidean space and a Hermitian symmetric space of
noncompact type can be relatives.

Denote the Euclidean metric on $\CC^n$ by  $\omega_{\CC^n}$. For each $1 \leq j \leq J$, let the bounded symmetric domain $\Omega_j \subset \CC^{m_j}$ be the Harish-Chandra realization of an irreducible Hermtian symmetric space of noncompact type and let $\omega_{\Omega_j}$ be the Bergman metric on $\Omega_j$. Let $D \subset \CC^\kappa$ be a connected open set and $\omega_D$ be a K\"ahler metric on $D$, not necessarily complete. 

 In this short paper, we show that  there do not simultaneously exist  holomorphic isometric immersions
  $F: (D, \omega_D) \rightarrow (\CC^n, \omega_{\CC^n})$ and $G=(G_1, \cdots, G_J): (D, \omega_D) \rightarrow (\Omega_1, \mu_1 \omega_{\Omega_1}) \times \cdots \times (\Omega_J, \mu_J \omega_{\Omega_J})$ with $\mu_1, \cdots, \mu_J$ positive real numbers.
 As a consequence,
  a complex Euclidean space and a bounded symmetric domain cannot be relatives.
Indeed, we prove the following slightly stronger result:

\begin{theorem} \label{main1}
Let $D \subset \CC$ be a connected open subset. 
Suppose that
$F: D \rightarrow \CC^n$ and $G=(G_1, \cdots, G_J): D \rightarrow
\Omega=\Omega_1 \times \cdots \times \Omega_J$ are holomorphic
mappings such that

\begin{equation}\label{isometry}
F^*\omega_{\CC^n} =\sum_{j=1}^J \mu_j G_j^*\omega_{\Omega_j} ~~\text{on}~~D
\end{equation}
 for certain real  constants $\mu_1,
\cdots, \mu_J$. Then $F$ must be a constant map. Furthermore, if all
$\mu_j's$ are positive, then  $G$ is also a constant map.
\end{theorem}

\begin{corollary}\label{corollary}
There does not exist a K\"ahler manifold $(X, \omega_X)$ that can be
holomorphic isometrically embedded into the complex Euclidean space
 $(\CC^n, \omega_{\CC^n})$ and  also  into  a Hermitian symmetric space of noncompact type $(\Omega, \omega_\Omega)$.

\end{corollary}


\bigskip

{\bf Acknowledgement}: We thank  Di Scala for helpful communication
related to this work. Indeed, this short paper is motivated by the
question raised in his communication, that is answered by Corollary \ref{corollary}.

\bigskip

\bigskip

\section{Proof of Theorem \ref{main1}}

Our  proof fundamentally uses  ideas developed in  our previous work
\cite{HY}. Let $D$ be  a domain in ${\mathbb C}$. Let $F=(f_1,
\cdots, f_n): D \rightarrow \CC^n,\ \  G=(G_1, \cdots, G_J): D
\rightarrow \Omega_1 \times \cdots \times \Omega_J$ be holomorphic
maps satisfying equation (\ref{isometry}). Without loss of
generality, assume that $0 \in D$ and  $F(0)=0, G(0)=0$. We argue by
contradiction by assuming that $F$ is not constant. 
By equation (\ref{isometry}), we  have

$$\partial\bar\partial \left(\sum_{i=1}^n |f_i(z)|^2 \right)= \sum_{j=1}^J \mu_j \partial\bar\partial \log K_j(G_j(z), \overline{G_j(z)})~~\text{for~~}z \in D,$$
where $K_j(\xi, \overline{\eta})= \sum_l h_{jl}(\xi)
\overline{h_{jl}(\eta)}$ is the Bergman kernel on $\Omega_j$ and
$\{h_{jl}(\xi)\}$ is an orthonormal  basis of $L^2$ integrable
holomorphic functions over $\Omega_j$. Note that $\Omega_j$ is a complete circular domain in the Harish-Chandra realization. Therefore,
the Bergman kernel of $\Omega_j$ satisfies the identity $K_j(e^{\sqrt{-1}\theta}\xi, \overline{e^{\sqrt{-1}\theta}\eta}) = K_j(\xi, \overline{\eta})$
 for any $\theta \in \RR$ and any $\xi, \eta \in \Omega_j$. This implies  $K_j(e^{\sqrt{-1}\theta}\xi, 0) = K_j(\xi, 0)$. Therefore $K_j(\xi, 0)$ is a positive
  constant. In another word, $K_j(\xi, \overline{\eta})$ does not contain any nonconstant pure holomorphic terms in $\xi$. Similarly,  $K_j(\xi, \overline{\eta})$
  does not contain any nonconstant pure anti-holomorphic terms in $\eta$. Hence $K_j(\xi, \overline{\xi})$ does not contain nonconstant pluriharmonic
   terms in $\xi$. After normalization, we can assume tha $K_j(\xi, 0) =1$.
By the standard argument in \cite{CU}, one can get rid of $\partial\bar\partial$ to obtain the following functional identity by comparing the
 pure holomorphic and anti-holomorphic terms in $z$:

\begin{equation}\label{function}
\sum_{i=1}^n |f_i(z)|^2 = \sum_{j=1}^J \mu_j \log K_j(G_j(z),
\overline{G_j(z)}) ~~\text{~for~any~}z \in D.
\end{equation}
%
After polarization, (\ref{function}) is equivalent to
\begin{equation}\label{polarization}
\sum_{i=1}^n f_i(z) \bar{f_i}(w) =\sum_{j=1}^J \mu_j \log
K_j(G_j(z), \bar{G_j}(w)) ~~\text{~for~~}(z, w) \in D\times
\hbox{conj}({D}),
\end{equation}
where $\bar{f}_i(w) = \overline{f_i(\overline{w})}$ and
$\hbox{conj}({D})=\{z \in \CC | \bar z \in D\}$. Notice  that the
Bergman kernel $K_j(\xi, \eta)$ is a rational function on $\xi$ and
$\eta$ for the bounded symmetric domain $\Omega_j$ (\cite{FK}). From
this, we  have the following algebraicity lemma. Here, we recall
that a function $H$ is called a holomorphic Nash algebraic function
over $V \subset {\mathbb C}^\kappa$ if $H$ is holomorphic over $V$
and there is a non-zero polynomial $P(\eta, X)$ in $(\eta, X)$ such
that $P(\eta,H(\eta))\equiv 0$ for $\eta \in V$.

\begin{lemma}\label{algebraicity}
For any $1\leq i \leq n$, $f_i(z)$ can be written as a holomorphic
Nash algebraic function in $G(z)=(G_1(z), \cdots, G_J(z))$, after
shrinking $D$ toward the origin if needed.
\end{lemma}

\begin{proof}
The proof is similar  to the algebraicity lemma  in Proposition 3.1
of \cite{HY}.
Write
$D^\delta={\partial^\delta \over\partial w^\delta}$. Applying   the
 differentiation $\partial \over\partial w$ to equation (\ref{polarization}), we get for $w$ near $ 0$ the following:
\begin{equation}\label{james}
\sum_{i=1}^{n} f_i(z) {\partial \over\partial w} \bar {f_i}(w)=
\sum_{j=1}^{J} \mu_j \frac{{\partial \over\partial w} K_j(G_j(z),
\bar{G_j}(w))} {K_j(G_j(z), \bar{G_j}(w))}.
\end{equation}
We can rewrite (\ref{james}) as
follows:
\begin{equation}\label{james-3}
F(z) \cdot D^1(\bar F(w)) = \phi_1(w,  {G_1}(z), \cdots, {G_J}(z)),
\end{equation}
where $F=(f_1, \cdots, f_n)$, and  $\phi_1(w, X_1, \cdots, X_J)$ is
Nash algebraic  in $(X_1,\cdots,X_J)$ for each fixed  $w$,  as the
Bergman kernel functions $K_j(\xi, \overline\eta)$ are rational
functions. Now, differentiating (\ref{james-3}), we get for any
$\delta$ the following equation

\begin{equation}
\label{james-007-4} F(z) \cdot D^{\delta}(\bar F(w))
=\phi_{\delta}(w, G_1(z), \cdots, G_J(z)).
\end{equation}
Here for $\delta> 0$,
$\phi_{\delta}(w, X_1, \cdots, X_J)$ 
is Nash algebraic in
$X_1, \cdots, X_J$ for any fixed $w$. 

Now, let ${\mathcal
L}:=\hbox{Span}_{\mathbb C}\{ D^{\delta}(\bar F(w))|_{w=0}\}_{\delta\ge
1}$ be a vector subspace of ${\mathbb C}^{n}$. Let
$\{D^{\delta_j}(\bar F(w))|_{w=0}\}_{j=1}^{\tau}$ be a basis for
$\mathcal L$. Then for a small open disc $\Delta_0$ centered at $0$ in
${\mathbb C}$, $\bar F(\Delta_0)\subset {\mathcal L}.$ Indeed, for any
$w$ near $0$, we have from the Taylor expansion that
$$\bar F(w)=\bar F(0)+\sum_{ \delta \ge 1}{D^{\delta}(\bar F)(0)\over\delta
!}w^{\delta}=\sum_{ \delta \ge 1}{D^{\delta}(\bar F)(0)\over\delta
!}w^{\delta}\in {\mathcal L}.$$

Now, let $\nu_j
~(j=1 \cdots, n-\tau)$ be a basis of the Euclidean orthogonal
complement of ${\mathcal L}.$  Then, we have

\begin{equation}
\label{james-007-06} F(z)\cdot \nu_{j} =0, ~~ \text{for~each}~~ j=1,\cdots,
n-\tau.
\end{equation}
Consider the system consisting of (\ref{james-007-4}) at $w=0$ (with
$\delta=\delta_1,\cdots,\delta_\tau$) and (\ref{james-007-06}).
The linear coefficient matrix in the left hand side of
the system at $w=0$ with respect to $F(z)$ is
\begin{equation}\notag
\begin{bmatrix}
D^{\delta_1}(\bar F(w))|_{w=0}\\
\vdots\\
D^{\delta_\tau}(\bar F(w))|_{w=0}\\
\nu_1\\
\vdots\\
\nu_{n-\tau}
\end{bmatrix}
\end{equation}
and is obviously invertible. Note that the right hand side of the
system of equations consisting of (\ref{james-007-4}) at $w=0$ (with
$\delta=\delta_1,\cdots,\delta_\tau$) and
is Nash algebraic in  
$G_1(z), \cdots, G_J(z)$. By  Gramer's rule, there exists a Nash
algebraic function $\hat F(X_1, \cdots, X_J)$ in all variables $X_1,
\cdots, X_J$ such that $F(z)=\hat F (G_1(z), \cdots, G_J(z))$ near
$z=0$. In fact, in our setting here, we can make $\hat{F}$
holomorphically rational in its variables.
\end{proof}

Let $G=(G_1, \cdots, G_J) = (g_{11}, \cdots, g_{1 m_1}, \cdots, g_{J1}, \cdots, g_{J m_J})$.   Let $\mathfrak{R}$ be the field of  rational functions
in $z$ over $D$.  Consider the field
  extension $$\mathfrak{F}=\mathfrak{R}( g_{11}(z), \cdots, g_{J m_J}(z) ),$$
  namely, the smallest subfield of meromorphic function field over $D$ containing  rational functions and $g_{11},\cdots, g_{J m_J}$.
    Let $l$ be the transcendence degree of the field extension
   $\mathfrak{F} / \mathfrak{R}$.

 \medskip

   If $l = 0$, then each element in $\{g_{11}(z), \cdots, g_{Jm_J}(z) \}$ is a Nash algebraic function. Hence by Lemma \ref{algebraicity}, each $f_i(z)$ is
   also Nash algebraic. In this case, we arrive at a contradiction by the following lemma together with Equation (\ref{polarization}).

   \begin{lemma} Let $V \subset \CC^\kappa$ be a connected open set. Let $H_1(\xi_1, \cdots, \xi_\kappa), \cdots,  H_K(\xi_1, \cdots, \xi_\kappa)$
    and $H(\xi_1, \cdots, \xi_\kappa)$  be  holomorphic Nash algebraic functions on $V$.
    Assume that
   $$\exp^{H(\xi_1, \cdots, \xi_\kappa)} =  \prod_{k=1}^K \left( H_k(\xi_1, \cdots, \xi_\kappa) \right)^{\mu_k} \ \xi\in V,$$ for certain
    real numbers $\mu_1, \cdots, \mu_K$. Then
    $H(\xi_1,\cdots,\xi_\kappa)$ is constant.
   \end{lemma}

   \begin{proof} Suppose that $H$ is not constant. After a linear transformation in $\xi$, if needed, we can assume, without loss of generality,
    that,  $H(\xi)$ is not  constant for a certain fixed $\xi_2, \cdots, \xi_\kappa$.
    Then $H$ is a non-constant  Nash-algebraic holomorphic   function in $\xi_1$ for such fixed  $\xi_2, \cdots, \xi_\kappa$.
     Hence, we can assume  that $\kappa=1$ to achieve a contradiction.
 Write $H=H(\xi)$ and $H_k=H_k(\xi)$ for each $1 \leq k \leq \kappa$. Use  $S \subset \CC$ to denote the union  of branch points, poles and zeros  of
     $H(\xi)$ and
     $H_k(\xi)$ for each $k$.  Given a  $p \in \CC \setminus S$ and a real curve in $\CC \setminus S$ connecting $p$ and $V$,
     by holomorphic continuation, the following equation holds on an open neighborhood of the curve:
   \begin{equation}\label{trans}
  \exp^{H(\xi)}  =\prod_{k=1}^K \left( H_k(\xi) \right)^{\mu_k}.
  \end{equation}
Assume that the minimal  polynomial of $H$ is given by $p(\xi, X)=
A_d(\xi)X^d + \cdots + A_0(\xi)$ such that $p(\xi, H(\xi)) \equiv
0$. Denote the branches of $H$ by $\{H^{(1)}, \cdots, H^{(d)}\}$ and
these branches can be obtained through $H$ by holomorphic
continuation.  Denote the corresponding branches for $H_k$ obtained
by holomorphic continuation by $\{H_k^{(1)}, \cdots, H_k^{(d)}\}$.
Let $\xi_0$ be a zero of $\frac{A_d}{A_0}$ or $\xi_0=\infty$ if
$\frac{A_d}{A_0}$ is a constant. Then some branches of $H$ blow up
at $\xi_0$. Without loss of generality, assume that $\xi_0=\infty$.
Assume that (\ref{trans}) holds in a neighborhood of $\infty$ after
holomorphic continuation from the original equality. By the Puiseux
expansion, we can assume that $$H(\xi)=\sum_{\beta=\beta_0,
\beta_0-1,\cdots,
-\infty}a_{\beta}\xi^{\beta/N_0}=a_{\beta_0}\xi^{\beta_0/N_0}+o(|\xi|^{\beta_0/N_0})$$
for $|\xi|>>1$ with $a_{\beta_0}\not = 0$ and $\beta_0, N_0 >0$.
Without loss of generality, we assume that  $a_{\beta_0}>0$. Now,
when $\xi\ra \infty $ along the positive $x$-axis, for the branch
$H^{(*)}$, which corresponds to $\xi^{\beta_0/N_0}$ taking positive
value along this ray in its Puiseux expansion, we have
 $|e^{H^{(*)}(x)}|\ge e^{\left(\frac{a_{\beta_0}}{2}x^{\frac{\beta_0}{N_0}}\right)}$ 
 as $x\ra +\infty$. However, the
 right hand side of (\ref{trans}) grows at most
 polynomially. This is a contradiction.
 \end{proof}

\bigskip

Now, assume that $l>0$.  By re-ordering the lower index, let
$\mathcal{G}= \{g_1(z), \cdots, g_l(z)\}$ be the maximal algebraic
  independent subset in $\mathfrak{F}$. It follows that the transcendence degree of $\mathfrak{F} / \mathfrak{R}(\mathcal{G})$ is 0.
  Then   there exists a small connected open subset $U$ with $0\in
  \overline{U}$ such that for each $j_\a$ with $g_{j_\a}\not \in \mathcal{G}$,
   we have a holomorphic Nash algebraic function $\hat g_{j_ \alpha}(z, X_1, \cdots, X_l) $ in the neighborhood
   $\hat U$ of $\{(z, g_1(z), \cdots, g_l(z)) | z\in U\}$ in $\CC\times\CC^l$ such
   that it holds that
    $g_{j_ \alpha}(z) = \hat g_{j \alpha} (z, g_1(z), \cdots, g_l(z))$ for any $z \in U$.
    Then by Lemma \ref{algebraicity}, for each $1\leq i \leq n$, there exists  a holomorphic Nash algebraic function $\hat f_i(z, X_1, \cdots, X_l)$ in $\hat U$
   such that $f_i(z) = \hat f_i(z, g_1(z), \cdots, g_l(z))$ for  $z \in U$. 
Define $$\Psi(z, X, w) = \sum_{i=1}^n \hat f_i(z, X) \bar f_i(w)-
\sum_{j=1}^J \mu_j \log K_j(\cdots, X_\gamma, \cdots, \hat g_{j
\alpha}(z, X), \cdots, \bar{g}_{j1}(w), \cdots, \bar{g}_{jm_j}(w))$$
and
$$\Phi (z, X, w)= \frac{\partial}{\partial w} \Psi (z, X, w)$$ for $(z, X, w) \in \hat U \times \hbox{conj}(U)$,
where $X=(X_1, \cdots, X_l)$.

\begin{lemma}
For any $w$ near 0 and any $(z, X)\in \hat U$, $\Phi(z, X, w) \equiv 0$. As a consequence, $\Psi(z, X, w) \equiv 0$.
\end{lemma}
\begin{proof}
Assume $\Phi(z, X, w) \not\equiv 0$. Then there exists $w_0$ near 0, such that $\Phi(z, X, w_0) \not\equiv 0$. Since $\Phi(z, X, w_0)$ is a Nash algebraic function in $(z, X)$, then there exists a holomorphic polynomial $P(z, X, t)=A_d(z, X)t^d + \cdots + A_0(z, X)$ of degree $d$ in $t$, with $A_0(z, X) \not\equiv0$ such that $P(z, X, \Phi(z, X, w_0)) \equiv 0$.

As $\Psi(z, g_1(z), \cdots, g_l(z), w) \equiv 0$ for $z\in {U}$, it
follows that $\Phi(z, g_1(z), \cdots, g_l(z), w_0) \equiv 0$ and
therefore $A_0(z, g_1(z), \cdots, g_l(z)) \equiv 0$. This means that
$\{g_1(z), \cdots, g_l(z)$\} are algebraic dependent over
$\mathfrak{R}$. This is a contradiction.

Since $\Psi(z, X, w)$ is holomorphic in $w$ and $\Psi(z, X, 0) \equiv 0$, then $\Psi(z, X, w) \equiv 0$.
\end{proof}

Now for any $(z, X, w) \in \hat U \times \hbox{conj}(U)$, we have
the following functional identity:
\begin{equation}\label{zero}
\sum_{i=1}^n \hat f_i(z, X) \bar f_i(w) =\sum_{j=1}^J \mu_j \log
K_j(\cdots, X_{\gamma}, \cdots, \hat g_{j \alpha}(z, X), \cdots,
\bar{g}_{j1}(w),
 \cdots, \bar{g}_{jm_j}(w)).\end{equation}

\begin{lemma}\label{nonconstant}

There exists $(z_0, w_0) \in U \times \hbox{conj}(U)$ such that
\begin{equation}\notag
\sum_{i=1}^n \hat f_i(z, X) \bar f_i(w) \not\equiv 0.\end{equation}
\end{lemma}

\begin{proof}
Assume not. Letting $w=\overline{z}$ and $X=(g_1(z), \cdots, g_l(z))$, 
it follows that

$$\sum_{i=1}^n |f_i(z)|^2 = \sum_{i=1}^n \hat f_i(z, g_1(z), \cdots, g_l(z)) \bar f_i(\overline{z}) \equiv 0,\ \hbox{over}\ U.$$ This implies that $f_i(z) \equiv 0$ for all $1 \leq i \leq n$ and therefore contradicts to the assumption that $F=(f_1, \cdots, f_n)$ is a non-constant map.
\end{proof}

Choosing $z_0, w_0$ as in Lemma  \ref{nonconstant}, 
$\sum_{i=1}^n \hat f_i(z, X) \bar f_i(w)$ is a nonconstant
holomorphic Nash algebraic function in $X$ by Lemma
\ref{nonconstant} and by the fact that $$K_j(\cdots, X_{\gamma},
\cdots, \hat g_{j \alpha}(z, X), \cdots, g_{j1}(w), \cdots,
g_{jm_j}(w))$$ is also Nash algebraic in $X$ for all $j$ as the
Bergman kernel function of bounded symmetric domain is rational.
Hence we arrive at a contradiction by Lemma \ref{algebraicity}. Thus
$F$ must be a constant map.
Now if all $\mu_j's$ are further assumed to be positive, it is
obvious that $G$ must also be constant.
The proof of Theorem \ref{main1} is  complete.


\section{Further remarks}

A  Hermitian symmetric space $M$ of compact type can be
holomorphically isometrically embedded into the complex projective
space $\PP^N$ by the Nakagawa-Takagi embedding.
 Notice that the Fubini-Study metric $\omega_{\PP^N}$ on $\PP^N$ in a standard holomorphic chart $\{w_1, \cdots, w_N\}$ is given by $$\omega_{\PP^N} =
 \sqrt{-1}\partial\bar\partial \log(1+ \sum_j |w_j|^2) $$ up to the normalizing constant, which is also of the form $\partial\bar\partial \log K(w, \bar{w})$,
  where $K(w, \bar{w})$ is an algebraic function. Therefore the same argument yields the following theorem:

\begin{theorem} \label{main2}
Let $D \subset \CC$ be a connected open subset. 
If there are  holomorphic maps 
$F: D \rightarrow \CC^n$ and $G=(G_1,
\cdots, G_J): D \rightarrow \Omega_1 \times \cdots \times \Omega_J$ and $L = (L_1, \cdots, L_K): D \rightarrow \PP^{N_1} \times \cdots \times \PP^{N_K}$
such that

\begin{equation}\notag
F^*\omega_{\CC^n} =\sum_{j=1}^J \mu_j G_j^*\omega_{\Omega_j} + \sum_{k=1}^K \lambda_k {L_k}^*\omega_{\PP^{N_k}} ~~\text{on}~~D
\end{equation}
 for   real constants $\mu_1,
\cdots, \mu_J, \lambda_1, \cdots, \lambda_K$. Then $F$ is a constant
map. Moreover, if $\mu_j,\lambda_j$ are positive, $G$ and $L$ are
also constant map.
\end{theorem}

Remark that the above constant $\mu_1, \cdots, \mu_J, \lambda_1,
\cdots, \lambda_K$ can be positive, negative or zero.
 In particular, Theorem \ref{main2} implies that the complex Euclidean space cannot be a relative to the product space of a bounded symmetric space and a Hermitian symmetric space of compact type. 
Note that, in \cite{DL2}, Di Scala and Loi showed that any bounded domain with Bergman metric and a  Hermitian symmetric space of compact type cannot be relatives. Combining their results, we actually can conclude that any Hermitian symmetric space of a particular type and the product
 of Hermitian symmetric spaces of two other types cannot be relatives. More precisely, we summarize the result as follows:

\begin{theorem} \label{main3}
Let $D \subset \CC$ be a connected open set. Let $\Omega, M, \CC^n$ be a Hermitian symmetric space
 of noncompact, compact and Euclidean type, respectively, equipped with the canonical metrics $\omega_\Omega, \omega_M, \omega_{\CC^n}$. 
If there exist non-constant holomorphic maps 
$F: D \rightarrow \CC^n$, $G: D \rightarrow \Omega$ and $L: D \rightarrow M$
such that

\begin{equation}\notag
a F^*\omega_{\CC^n} + b G^*\omega_{\Omega} + c {L}^*\omega_{M} =0~~\text{on}~~D
\end{equation}
 for real constants $a, b, c$, then it must holds that $a = b= c=0$.
 \end{theorem}

Next, we let $(D_1\subset {\mathbb C}^n,\omega_1)$ and  $(D_2\subset
{\mathbb C}^m,\omega_2)$ be two K\"ahler manifolds with
$\omega_j=\sqrt{-1}\partial\bar{\partial} \log h_j(z,\bar{z})$. Here
$h_j(z,\bar{z})$ are real analytic functions in $z$. Assume that
$0\in D_j$ and $h_j$ ($j=1,2$) do not have any non-constant harmonic
terms in its Taylor expansion at the origin with $h_j(0,0)$ being
normalized to be $1$. $(D_1,\omega_1)$ and $(D_2,\omega_2)$ are
relative at $0$ if and only if there are non-constant holomorphic
maps $\phi_1: \Delta\ra D_1$ and $\phi_2: \Delta\ra D_2$ with
$\phi_j(0)=0$ such that $\phi_1^{*}(\omega_1)=\phi_2^{*}(\omega_2)$.
Here $\Delta$ is the unit disk in ${\mathbb C}^1$. As standard, this
happens if and only if
$$h_1(\phi_1(\tau),\overline{\phi_1(\tau)})=h_2(\phi_2(\tau),\overline{\phi_2(\tau)}).$$
Now, we let the real analytic set $M\subset D_1\times D_2\subset
{\mathbb C}^{n+m}$ be defined by $ h_1(z,\-{z})=h_2(w,\-{w})$ with $
(z,w)\in D_1\times D_2$. By the fact that $h_j$ serve as potential
functions of K\"ahler metrics near $0$, it is not hard to show that
$M$ must be regular at the origin. Then $(D_1,\omega_1)$ and
$(D_2,\omega_2)$ are relative at $0$ or near a point close to $0$ if
and only if inside $M$, there is a non-trivial  holomorphic curve
containing the origin.  Then this cannot happen if and only if $M$
is of D'Angelo finite type at $0$ [DA]. Hence, by what we proved
above, we have the following:

\begin{theorem}\label{main4}
Let $K_j(w,\bar{w})$ ($j=1,\cdots, \kappa$) be postively-valued
smooth Nash-algebraic functions in $(w,\bar{w})$ with $w(\in
{\mathbb C}^m)\approx 0$. Assume that the complex Hessian of $\log
K_j(w,\-{w})$ is positive definite for each $j$. Then for any
positive  real numbers $\mu_1,\cdots,\mu_\kappa$, the following
real-analytic hypersurface $M$ defined near the origin is of finite
D'Angelo type at $0$:
$$M:=\{(z,w)(\subset {\mathbb C}^{n+m})\approx (0,0):\ \
\sum_{j=1}^{n}|z_j|^2=\sum_{l=1}^{\kappa}\mu_l\log
K_l(w,\bar{w})\}.$$
\end{theorem}



\begin{thebibliography}{99}


\bibitem[C]{C}Calabi, E.: {\em Isometric imbedding of complex manifolds}, Ann.
of Math. (2) 58, (1953). 1--23.


\bibitem[CU]{CU}Clozel, L. and Ullmo, E.: {\em Correspondances modulaires et
mesures invariantes}, J. Reine Angew. Math. 558 (2003), 47--83.


\bibitem [DA]{john} D'Angelo, J., Real hypersurfaces, orders of contact, and
applications, Ann. Math., 115 (1982), 615-637.



\bibitem[DIL]{DIL} Di Scala, A.; Ishi, H. and Loi, A.: {\em K\"ahler immersions of homogeneous K\"ahler manifolds into complex space forms}, Asian J. Math. 16 (2012), no. 3, 479-487.


\bibitem [DL1] {DL1}Di Scala, A. and Loi, A.: {\em K\"ahler maps of Hermitian
symmetric spaces into complex space forms,} Geom. Dedicata 125
(2007), 103-113.

\bibitem[DL2]{DL2}Di Scala, A. and Loi, A.: {\em K\"ahler manifolds and their relatives,} Ann. Sc. Norm. Super. Pisa Cl. Sci. (5) 9 (2010), no. 3, 495-501.



\bibitem[FK]{FK}Faraut, J. and Ko\'ranyi, A.: {\em Function spaces and reproducing kernels on bounded symmetric domains,} J. Funct. Anal. 88 (1990), 64-89.


\bibitem[Hu1]{Hu1}Huang, X.: {\em On the mapping problem for algebraic real
hypersurfaces in the complex spaces of different dimensions}, Ann.
Inst. Fourier (Grenoble) \textbf{44} (1994), no. 2, 433--463.


\bibitem[Hu2]{Hu2} Huang, X.: {\em Schwarz reflection principle in complex spaces of dimension two,} Comm. Partial Differential Equations 21 (1996), no. 11-12, 1781-1828.





\bibitem[HY]{HY} Huang, X. and Yuan, Y.: {\em Holomorphic isometry from a K\"ahler manifold into a product of complex projective manifolds}, Geom. Funct. Anal. 24 (2014), no. 3, 854-886.

\bibitem[Ji]{Ji} Ji, S.: {\em Algebraicity of real analytic hypersurfaces with maximal rank,} Amer. J. Math. 124 (2002), no. 6, 1083-1102.






\bibitem[Mo1]{Mo1} Mok, N.: {Metric Rigidity Theorems on Hermitian Locally Symmetric Manifolds}, Series in Pure Mathematics. 6. World Scientific Publishing Co., Inc., Teaneck, NJ, 1989. xiv+278 pp..

\bibitem[Mo2]{Mo2} Mok, N.: {\em Geometry of holomorphic isometries and related maps between bounded domains}, Geometry and analysis. No. 2, 225-270, Adv. Lect. Math. (ALM) 18, Int. Press, Somerville, MA, 2011.

\bibitem[Mo3]{Mo3} Mok, N.: {\em Extension of germs of holomorphic isometries
up to normalizing constants with respect to the Bergman metric}, J. Eur. Math. Soc. 14 (2012), no. 5, 1617-1656.






\bibitem[MN]{MN} Mok, N. and Ng, S.: {\em Germs of measure-preserving holomorphic maps from bounded symmetric domains to their Cartesian products}, J. Reine Angew. Math.  669 (2012), 47-73.





\bibitem[Ng]{Ng} Ng, S.: {\em Holomorphic Double Fibration and the Mapping Problems of Classical Domains,} to appear in Int. Math. Res. Not.






\bibitem[U]{U} Umehara, M.: {\em Kaehler submanifolds of complex space forms,} Tokyo J. Math. 10 (1987), no. 1, 203-214.



\bibitem[YZ]{YZ} Yuan, Y. and Zhang, Y.: {\em Rigidity for local holomorphic isometric embeddings from $B^n$ into $B^{N_1} \times ... \times B^{N_m}$ up to conformal factors},  J. Differential Geom. 90 (2012), no. 2, 329-349.

\bigskip

 \noindent Xiaojun Huang, huangx@math.rutgers.edu, Department of Mathematics, Rutgers University, Piscataway, NJ 08854 USA.\\

 \noindent Yuan Yuan, yyuan05@syr.edu, Department of Mathematics, Syracuse University, Syracuse, NY 13244 USA.\\




\end{thebibliography}
\end{document}